\documentclass[11pt]{amsart}
\usepackage{amsmath, amsthm, amscd, amsfonts, amssymb, graphicx, color,float,pgf,tikz}
\usepackage{amssymb,fontenc} 
\usepackage{latexsym,wasysym,mathrsfs}
\usepackage{hyperref}
\usetikzlibrary{arrows}
\makeatletter \oddsidemargin.4375in \evensidemargin \oddsidemargin
\newtheorem{theorem}{Theorem}[section]

\theoremstyle{definition}

\newtheorem{example}[theorem]{Example}

\newtheorem{remark}[theorem]{Remark}
\numberwithin{equation}{section}

\newcommand{\be}{\begin{equation}}
\newcommand{\ee}{\end{equation}}

\numberwithin{equation}{section}

\begin{document}

\noindent   
  \\[0.50in]


\title[ Inverse eigenvalues problem of nonnegative matrices $\cdots$]{Inverse eigenvalues problem of nonnegative matrices via unit lower triangular matrices }

\author[A. M. Nazari]{Alimohammad Nazari$^1$ }

\address{$^1$Department of Mathematics, Arak University, Arak, Iran. P. O. Box 38156-8943. }

\email{ a-nazari@araku.ac.ir }

\author[A. Nezami ]{ Atiyeh Nezami$^2$}

\address{$^2$Department of Mathematics, Arak University, Arak, Iran. P. O. Box 38156-8943.. }

\email{   a-nezami@arshad.araku.ac.ir  }



\keywords{Nonnegative matrices, Unit triangular matrices,  Inverse eigenvalue problem. \\
\indent Received: dd mmmm yyyy,    Accepted: dd mmmm yyyy.
\\
\indent $^{*}$ Corresponding author}
\maketitle
\hrule width \hsize \kern 1mm


\begin{abstract}
 The main goal of this work is to use the unit lower  triangular  matrices  for solving inverse eigenvalue problem of nonnegative matrices (NIEP) and present the easier method  to solve this  problem. We solve the problem for any given number of the real and complex eigenvalues.
Finally, we present the necessary and sufficient conditions to solve this problem with this method.   
 \end{abstract}
\maketitle
\vspace{0.1in}
\hrule width \hsize \kern 1mm

\section{Introduction And Preliminaries }
A matrix is called unit lower triangular  if it is lower triangular matrix and all entries on its main diagonal are one.
The inverse of these matrices also is unit lower triangular. In Gaussian elimination method and LU factorization  unit lower triangular matrices play a very important roll.

The  nonnegative inverse eigenvalue problem (NIEP) asks
for necessary and sufficient conditions on a list
$\sigma=(\lambda_1,\lambda_2, \ldots,\lambda_n)$ of real or complex numbers
in order that it be the spectrum of a  nonnegative
matrix $A$ with spectrum $\sigma$, we will say that $\sigma$ is
  realizable   and that it is  realization of $\sigma$.
  
   Some necessary conditions  on the list of real number $\sigma=(\lambda_1,\lambda_2, \ldots,\lambda_n)$
to be the spectrum of a  nonnegative matrix are listed
below.\\    
\begin{align}\label{JLL}
\text{(1) The Perron eigenvalue $\max \{|\lambda_i |; \lambda_i \in \sigma \}$
 belongs to $\sigma$ (Perron-Frobenius theorem).}\\
\text{ (2) $s_k=\sum_{i=1}^n \lambda_i^k \ge 0.$ }\hspace{11cm}\,\, \nonumber \\
\text{(3)$s_k^m \le n^{m-1}s_{km}$ for $k,m = 1, 2, \ldots$ (JLL
inequality)\cite{18, 14M}}.\hspace{4.65cm}\nonumber 
\end{align}    
Although many mathematicians have worked on the inverse eigenvalue problem of nonnegative matrices \cite{ 1M, soto, 2M, 3M,  5, 6M, 7M,   10, 11, 12, Suleimanova}, our aim of this paper is solving the problem in the easier method via similarity of a matrix with upper triangular matrix. 

Two matrix $A$ and $B$ is called similar if there exist an invertible matrix $P$ such that $P AP^{-1}=B$. We recall that two similar matrices have same eigenvalues and this Theorem plays a very important roll in this paper.   

From now on in this paper $\lambda_1$ is the symbol of  Perron eigenvalue, and assume that all of given list of real spectrum  satisfy in necessary conditions (1),(2) and (3). Our method is the following, started by Guo in \cite{Guo} and we continue Guo's method in this paper.

The paper is organized as follows: At first for a given real set of eigenvalues that the
number of its positive eigenvalues is less than or equal the number of negative eigenvalues,
we solve the NIEP via unit lower triangular matrix. In continue for a given set of eigen-
values as $\sigma$ with nonnegative summation, in which the number of negative eigenvalues  $\sigma$ less
than the number of negative eigenvalues, we find a nonnegative matrices such that  $\sigma$ is
its spectrum. In the any case of problem at first we solve the problem for a finite number
of eigenvalues  $\sigma$ and then we provide the general solution. Finally we present a Theorem
that shows the necessary and sufficient conditions for solving the problem in our way.
In section 3 we consider complex spectrum with Perron eigenvalue $\lambda_1$ and nonnegative
summation and again solve the NIEP.

\section{Real spectrum}
Let  $\sigma=\{\lambda_1,\lambda_2, \ldots,\lambda_n\}$ be a given spectrum such that $\lambda_1 \ge \lambda_2\ge \cdots \ge \lambda_k > 0 \ge \lambda_{k+1} \ge  \cdots\ge \lambda_n$. We divide the problem into two parts,   the first part is  $k \le \frac{n}{2}$ and the another part $k > \frac{n}{2}$ and solve the problem. 
\subsection{The spectrum with $k \le \frac{n}{2}$ }
In this section we study NIEP with condition $k \le \frac{n}{2}$.
\subsubsection{Spectrum with one positive eigenvalues}
\begin{theorem}   \cite{Suleimanova}.
Assume that given $ \sigma = \{\lambda_1, \cdots , \lambda_n \} $
such that $ \lambda_1 >0 \geq \lambda_2 \geq \cdots \geq \lambda_n $ and $ \sum_{i=1}^n \lambda_i \ge 0,$ then there exist a set of  nonnegative matrix that $ \sigma $ is its spectrum. 
\end{theorem}
{\bf Proof}. Let $n=2$, then consider the  upper triangular matrix $A=\left[ \begin {array}{ll} \lambda_{{1}}&\alpha_2\\ 0&
\lambda_{{2}}\end {array} \right]
$ and $L=\left[ \begin {array}{cc} 1&0\\ 1&1\end {array}
 \right]$, therefore the following matrix
  $$C=LAL^{-1}=\left[ \begin {array}{ll} \lambda_{{1}}-\alpha_2&\alpha_2
\\\noalign{\medskip}\lambda_{{1}}-\alpha_2-\lambda_{{2}}&\alpha_2+\lambda_
{{2}}\end {array} \right],
$$
has eigenvalues $\lambda_1$ and $\lambda_2$ and if $-\lambda_2 \le \alpha_2 \le \lambda_1$, then the matrix $C$ is nonnegative. 

For $n=3$ we consider $A=\left[ \begin {array}{lll} \lambda_{{1}}&\alpha_2&\alpha_3\\ 0
&\lambda_{{2}}&0\\ 0&0&\lambda_{{3}}\end {array}
 \right] $ and $L=\left[ \begin {array}{ccc} 1&0&0\\ 1&1&0
\\ 1&0&1\end {array} \right] 
$, then the matrix 
$$
C=LAL^{-1}=\left[ \begin {array}{lll} 
\lambda_1-\alpha_2-
\alpha_3&\alpha_2&\alpha_3\\ 
\lambda_1-\alpha_2- \lambda_2-\alpha_3&\alpha_2+ \lambda_2 & \alpha_3\\
 \lambda_1-\alpha_2-\alpha_3-\lambda_3&\alpha_2&\alpha_3+\lambda_3\end {array} \right],
$$
is similar to the matrix $A$ and if  $-\lambda_2\le \alpha_2 $, $-\lambda_3 \le \alpha_3$ and $\alpha_2 + \alpha_3 \le \lambda_1$ then the matrix $C$ is nonnegative.

In continue, we consider
$$A=\left[ \begin {array}{lllll} 
\lambda_{1}&\alpha_2&\alpha_3&\cdots & \alpha_n\\ 0
&\lambda_{{2}}&0&\cdots & 0 \\
&& \ddots &&\\
0&0 & 0 &\ddots  &0\\
0 & 0 & 0&\cdots &\lambda_{{n}}\end {array}
 \right], 
 $$ 
 and 
 $$
 L=\left[\begin{array}{lllll}
1 & 0 & 0 & \cdots & 0 \\ 
1 & 1 & 0 & \cdots & 0 \\ 
  &  & \ddots &   &  \\ 
1 & 0 &0 & \ddots & 0 \\ 
1 & 0 & 0 & \cdots & 1
\end{array}\right],
$$
then the matrix 
\begin{equation}\label{e1}
C=LAL^{-1}=\left[ \begin {array}{lllll} \lambda_{{1}}-t  &\alpha_2&\alpha_3 & \cdots & \alpha_n\\ 
\lambda_{1}-\lambda_2-t & \alpha_2+ \lambda_2  &\alpha_3 & \cdots & \alpha_n \\
&& \ddots &&\\
\lambda_{ 1 } -\lambda_{n-1}-t&\alpha_2& \alpha_3  &\ddots &\alpha_n\\
\lambda_{ 1 } - \lambda_n -t &\alpha_2 & \alpha_3 &\cdots &\alpha_n +\lambda_{n}
\end {array} \right], 
\end{equation}
that $t = \sum_{i=2}^{n} \alpha_i$ is similar  to the matrix $A$,
and if  
\begin{align}\label{se1}
-\lambda_i & \le \alpha_i, \qquad i=2, 3, \cdots , n,\nonumber\\
 t & \le \lambda_1. 
\end{align} 
then the matrix $C$ is nonnegative and has eigenvalues $\{\lambda_1, \lambda_2, \cdots, \lambda_n\}$.
\begin{remark}
If given the spectrum $\sigma$ with   zero summation, then in the above Theorem it is necessary that we lie the value $\lambda_i$ on the $(i,i)$  entry of main diagonal of the matrix $A$ for $i=1,2,\cdots,n$  and   $-\lambda_i$ in the entry $(1,i)$  of this matrix  for $i=1,2,\cdots,n$ and then the all elements on main diagonal of matrix $C$ are zero. 
\end{remark}
\begin{remark}
In Theorem 2.1 if $-\sum_{i=2}^n \lambda_i \le \lambda_1$ and we set $\alpha_i = - \lambda_i $ for $i= 2, 3, \cdots, n$, then we can construct a nonnegative matrix that all elements of its main diagonal are zero except elements of the first row and the fist column.
For instance if $\sigma = \{10 , -2, -2, -2, -1, -1 \}$, then by \eqref{e1} the following matrix has spectrum $\sigma $,
\[
C=\left [ \begin{array}{cccccc}
2  & 2 & 2 & 2 & 1 & 1 \\ 
4 & 0 & 2 & 2 & 1 & 1 \\ 
4 & 2 & 0 & 2 & 1 & 1 \\ 
4 & 2 & 2  & 0 & 1 & 1 \\ 
3 & 2 & 2 & 2 & 0 & 1 \\ 
3 & 2 & 2 & 2 & 1 & 0
\end{array} \right ].
\]
that $\alpha_i$ for $i=2, \cdots n$ satisfy  in the conditions  \eqref{se1}.  In this case we can determine all elements that lie on the main diagonal of matrix $C$ that summation of them is equal to the summations of eigenvalues of the matrix $A$, for instance 
if  we want  the first   four elements of main diagonal of matrix  \eqref{e1} are $\frac{1}{2}$, 
we must consider $\alpha_2=2.5, \alpha_3 = 2.5, \alpha_4=2.5, \alpha_5= 1, \alpha_6 = 1$
 then $t= 9.5 $ that satisfy in conditions \eqref{se1}, and   and the matrix $C$ is as below
\[
 \left[ \begin {array}{cccccc}  0.5& 2.5& 2.5& 2.5&1&1
\\ \noalign{\medskip} 2.5& 0.5& 2.5& 2.5&1&1\\ \noalign{\medskip} 2.5&
 2.5& 0.5& 2.5&1&1\\ \noalign{\medskip} 2.5& 2.5& 2.5& 0.5&1&1
\\ \noalign{\medskip} 1.5& 2.5& 2.5& 2.5&0&1\\ \noalign{\medskip} 1.5&
 2.5& 2.5& 2.5&1&0\end {array} \right] 
\]
\end{remark}
\subsubsection{Spectrum with two  positive eigenvalues}

In this section, at first we consider $ \sigma = \{\lambda_1, \lambda_2, \lambda_3, \lambda_4 \}$ such that $ \lambda_1 \geq \lambda_2 > 0 \ge \lambda_3 \geq \lambda_4 $
and $ \sum \lambda_i \geq 0$, $ \lambda_1 \geq |\lambda_i|, i= 3,4 $ and in continue for a given set $\sigma$ with two positive eigenvalues and three negative eigenvalues solve the problem and finally for two positive eigenvalues and  more than three negative eigenvalues with nonnegative summation  again study the problem.
\begin{theorem}
 Let $ \sigma = \{\lambda_1, \lambda_2, \lambda_3, \lambda_4 \}$ such that $ \lambda_1 \geq \lambda_2 > 0 \ge \lambda_3 \geq \lambda_4 $
and let $ \sum \lambda_i \geq 0$, $ \lambda_1 \geq |\lambda_i|, i= 3,4 $, then there exist a set of nonnegative matrices that $ \sigma $ is its spectrum.
\end{theorem}
{\bf Proof}. We start with the following matrices
$$
A=\left[ \begin {array}{llll} \lambda_{{1}}&\alpha_{{2}} +\alpha_{{4}
}&\alpha_{{3}}&0\\ 
0&\lambda_{{2}}&\alpha&\alpha_{{4}}\\ 
0&0&\lambda_{{3}}&0\\ 
0&0&0&\lambda_{{4}}
\end {array} \right],
\qquad
L=\left[ \begin {array}{cccc} 1&0&0&0\\1&1&0&0
\\ 1&0&1&0\\
1&1&0&1\end {array}
 \right].
$$
Then the following matrix 
\begin{equation}\label{e2}
C=LAL^{-1}=\left[ \begin {array}{llll} 
 \lambda_1 - t &\alpha_2+\alpha_4 &\alpha_3&0\\
 \lambda_1 - \lambda_2 - t  -\alpha &\alpha_2 +\lambda_2 &\alpha_3+\alpha &\alpha_4\\
 \lambda_1- \lambda_3 - t &\alpha_2 +\alpha_4 &\alpha_3 +\lambda_3 &0\\
  \lambda_1 -\lambda_2 - t -\alpha &\alpha_2 +\lambda_2 -\lambda_4 &\alpha_3 +\alpha &\alpha_4 +\lambda_4 
\end {array} \right], 
\end{equation}
that $t= \sum_{i=2}^{4} \alpha_i$ is similar to the matrix $A$.   If 
\begin{align}\label{se2}
-\lambda_i & \le \alpha_i, \qquad i= 2,3,4,\nonumber\\
t & \le \lambda_1 \nonumber \\
-\alpha_3 & \le \alpha \le \lambda_1 - \lambda_2 -t,
\end{align} 
then the  matrix $C$
 is nonnegative.
\begin{remark}
In above Theorem one of the interesting solution is $\alpha_i:=\lambda_i, \,\, i=1,2$ and $\alpha=\lambda_3$ that satisfies in condition  \eqref{se2} with  zero summation and  the all another elements of matrix $C$ are zero. For instance  let $\sigma = \{ 7, 3, -5, -5\}$, then $\sigma$ is realizable by following  nonnegative matrices
  \[
C =\left[ \begin {array}{cccc}
0&2&5&0 \\
2&0&0&5\\ 
5&2&0&0 \\ 
2&5&0&0
\end {array} \right],
 \]  
 this spectrum is studied in \cite{borobia} and we solve this problem easier. 
\end{remark}
Now we consider the set of $\sigma$ with two positive eigenvalues and three negative eigenvalues with special conditions.
\begin{theorem}
Consider spectrum $\sigma = \{\lambda_1, \lambda_2, \lambda_3, \lambda_4, \lambda_5  \}$
such that $\lambda_1 \ge \lambda_2 >0 > \lambda_3 \ge \lambda_4 \ge \lambda_5$ with nonnegative summation. If 
$t = \sum_{i=2}^{5} \alpha_i$ and there exist a nonnegative $5 \times 5$ matrix that $\sigma $ is its spectrum.
\end{theorem}
{\bf Proof}. In this case we consider 
\[
A= \left[ \begin {array}{ccccc} 
\lambda_1 &\alpha_2 +\alpha_4 +\alpha_5 &\alpha_3 &0&0\\
 0&\lambda_2 &\alpha &\alpha_4 &\alpha_5\\
 0&0&\lambda_3 &0&0\\ 
 0&0&0&\lambda_4 &0\\
 0&0&0&0&\lambda_5
\end {array} \right],
\]
where $\alpha_i \ge -\lambda_i, i=4,5$ and $\alpha_5 < \lambda_2 $ and also $\alpha_4+\alpha_5 \ge \lambda_2$
and 
\[
L= \left[ \begin {array}{ccccc} 
1&0&0&0&0\\
 1&1&0&0&0\\ 
 1&0&1&0&0\\ 
 1&1&0&1&0\\ 
 1&1&0&0&1
 \end {array} \right],
\]
then 
\begin{align}
\label{e3}
C=&L A L^{-1}\nonumber \\ =&
\left[ \begin {array}{lllll} 
\lambda_1-t&\alpha_2+\alpha_4+\alpha_5&\alpha_3&0&0\\ 
\lambda_1-\lambda_2-t-\alpha&\alpha_2+\lambda_2&\alpha_3+\alpha&
\alpha_4&\alpha_5\\
\lambda_1-\lambda_3-t&\alpha_2+\alpha_4+\alpha_5&\alpha_3+\lambda_3&0&0\\
\lambda_1 -\lambda_2-t-\alpha&\alpha_2+\lambda_2-\lambda_4&\alpha_3+
\alpha&\alpha_4+\lambda_4&\alpha_5 \\
\lambda_1 -\lambda_2-t-\alpha & \alpha_2+\lambda_2-\lambda_5 &\alpha_3+\alpha &\alpha_4&\alpha_5+\lambda_5
\end {array} \right].
\end{align}
The matrix $C$
is similar to the matrix $A$ and if satisfy the conditions 
\begin{align*}
-\lambda_i & \le \alpha_i, \qquad i= 2,3,4,5\nonumber\\
t & \le \lambda_1 \nonumber \\
-\alpha_3 & \le \alpha \le \lambda_1 - \lambda_2 -t,
\end{align*}
  then 
this matrix is nonnegative.
\begin{theorem}
Consider spectrum $\sigma = \{\lambda_1, \lambda_2, \cdots, \lambda_n  \}$
such that $\lambda_1 \ge \lambda_2 >0 \ge \lambda_3 \ge \cdots\ge \lambda_n$ with nonnegative summation.
 If   $t = \sum_{i=2}^{n} \alpha_i$ ,
 then there exist a set of  nonnegative $n \times n$ matrix that $\sigma $ is its spectrum.
\end{theorem}
{\bf Proof}. The proof is provided as the proof of previous  Theorem.  In this case we consider the matrices $A$ and $L$ respectively as
\[
A= \left[\begin{array}{cccccccc}
\lambda_1 & \alpha_2 + (\alpha_r + \cdots +\alpha_n)  & \alpha_3 & \cdots & \alpha_{r-1} & 0 & \cdots & 0 \\ 
0 & \lambda_2 & \beta_3 & \cdots & \beta_{r-1} & \alpha_r & \cdots & \alpha_n \\ 
0 & 0 & \lambda_3 & \cdots & 0 & 0 & \cdots & 0 \\ 
\vdots & \vdots & \vdots & \ddots & \vdots & \vdots & \cdots & \vdots \\ 
0 & 0 & 0 & 0 & \lambda_{r-1} & 0 & 0 & 0 \\ 
0 & 0 & 0 & 0 & 0 & \lambda_r & 0 & 0 \\ 
\vdots & \vdots & \vdots & \vdots & \vdots & \vdots & \ddots & \vdots \\ 
0 & 0 & 0 & \cdots & 0 & 0 & \cdots & \lambda_n
\end{array} 
\right],
\]
\[
L=\left[ \begin{array}{cccccccc}
1 & 0 & 0 & \cdots & 0 & 0 & \cdots & 0 \\ 
1 & 1 & 0 & \cdots & 0 & 0 & \cdots & 0 \\ 
1 & 0 & 1 & \cdots & 0 & 0 & \cdots & 0 \\ 
\vdots & \vdots & \vdots & \ddots & \vdots & \vdots & \cdots & \vdots \\ 
1 & 0 & 0 & \cdots & 1 & 0 & \cdots & 0 \\ 
1 & 1 & 0 & \cdots & 0 & 1 & \cdots & 0 \\ 
\vdots & \vdots & \vdots & \vdots & \vdots & \vdots & \ddots & \vdots \\ 
1 & 1 & 0 & \cdots & 0 & 0 & \cdots & 1
\end{array}  \right],
\]
where the values of the second row of the  matrix $A$, $ \alpha_i , i=3,\cdots,n$, must be at least equal to the values of the $-\lambda_i, i=3,\cdots, n$ values in their columns respectively, and the value of $\alpha_r$ is the last value in the second row  with starting of the last element on this row, such that $\sum^{n}_{i=r} \alpha_i \ge \lambda_2$.

Then we have 
\begin{align*}
C &=LAL^{-1}\\
&=\left[ \begin{array}{lllllllll}
c_{11} & c_{12} & \alpha_{3} & \alpha_{4} & \cdots & \alpha_{r-1} & 0 & \cdots & 0 \\ 
c_{21}& \alpha_2  + \lambda_2 & \alpha_{3}+\beta_{3} & \alpha_{4}+\beta_{4} & \cdots & \alpha_{r-1}+\beta_{r-1} & \alpha_r & \cdots & \alpha_n \\ 
c_{31} & c_{32} & \lambda_{3}+\alpha_{3} & \alpha_{4}+\beta_{4} & \cdots & \alpha_{r-1}+\beta_{r-1} & 0 & \cdots & 0 \\ 
c_{4, 1} & c_{42} & \alpha_{3}+\beta_{3} & \lambda_{4}+\alpha_{4} & \cdots & \alpha_{r-1}+\beta_{r-1} & 0 & \cdots & 0 \\ 
\vdots & \vdots & \vdots & \vdots & \ddots & \vdots & \vdots & \cdots & \vdots \\ 
c_{r-1, 1} &c_{r-1,2} & \alpha_{3}+\beta_{3} & \alpha_{4}+\beta_{4} & \cdots & \lambda_{r-1}+\alpha_{r-1} & 0 & \cdots & 0 \\ 
c_{r1} &c_{r2} & \alpha_{3}+\beta_{3} & \alpha_{4}+\beta_{4} & \cdots & \alpha_{r-1}+\beta_{r-1} & \lambda_r+\alpha_r & \cdots & \alpha_n \\ 
\vdots & \vdots & \vdots & \vdots & \vdots & \vdots & \vdots & \ddots & \vdots \\ 
c_{n1} & c_{n2} & \alpha_{3}+\beta_{3} & \alpha_{4}+\beta_{4} & \cdots & \alpha_{r-1}+\beta_{r-1} & 0 & \cdots & \lambda_n+\alpha_n
\end{array}  \right] 
\end{align*}
where
\begin{align*}
c_{11}=&\lambda_1 -t ,\\
 c_{31} =&c_{41} = \cdots = c_{r-1, 1}=\lambda_1-\lambda_2-t,\\
 c_{21}=&c_{r1}= \cdots = c_{n1}= \lambda_1-\lambda_2-t-(\beta_{3}+ \cdots + \beta_{r-1}),\\
c_{12} =& c_{32} = \alpha_2  + \alpha_r+ \cdots + \alpha_n,\\
c_{i2}=&\alpha_2  + \lambda_2 -\lambda_i, \qquad i=4 \cdots n,
\end{align*}
 that $t = \sum_{i=2}^{n} \alpha_i$. The matrix $C$ is nonnegative matrix if hold the following conditions 
  \begin{align*}
-\lambda_i & \le \alpha_i, \qquad i= 2,\cdots ,n\nonumber\\
t & \le \lambda_1 \nonumber \\
-\alpha_i & \le \beta_{i} \le \lambda_1-\lambda_2-t, \qquad i= 3,\cdots ,r-1,
\end{align*}  
and then  $\sigma$ is its spectrum. 
 \begin{example}
Let $\sigma =\{19,1,-5,-5,-3,-3,-2,-2 \}$. This spectrum is chosen  form \cite{soto2018} and by our method 
we again solve this problem.  We select two matrices $A$ and $L$ as:
\[
A=\left[ \begin {array}{cccccccc} 19&1&5&5&3&3&2&0\\\noalign{\medskip}0
&1&0&0&0&0&a_{{27}}&2\\\noalign{\medskip}0&0&-5&0&0&0&0&0
\\\noalign{\medskip}0&0&0&-5&0&0&0&0\\\noalign{\medskip}0&0&0&0&-3&0&0
&0\\\noalign{\medskip}0&0&0&0&0&-3&0&0\\\noalign{\medskip}0&0&0&0&0&0&
-2&0\\\noalign{\medskip}0&0&0&0&0&0&0&-2\end {array} \right],
\]
\[
L=\left[ \begin {array}{cccccccc} 1&0&0&0&0&0&0&0\\\noalign{\medskip}1&
1&0&0&0&0&0&0\\\noalign{\medskip}1&0&1&0&0&0&0&0\\\noalign{\medskip}1&0
&0&1&0&0&0&0\\\noalign{\medskip}1&0&0&0&1&0&0&0\\\noalign{\medskip}1&0
&0&0&0&1&0&0\\\noalign{\medskip}1&0&0&0&0&0&1&0\\\noalign{\medskip}1&1
&0&0&0&0&0&1\end {array} \right].
\]
Then the matrix $C=LAL^{-1}$ is computed as follows:
\[
C= \left[ \begin {array}{cccccccc} 0&1&5&5&3&3&2&0\\\noalign{\medskip}-1
-a_{{27}}&0&5&5&3&3&2+a_{{27}}&2\\\noalign{\medskip}5&1&0&5&3&3&2&0
\\\noalign{\medskip}5&1&5&0&3&3&2&0\\\noalign{\medskip}3&1&5&5&0&3&2&0
\\\noalign{\medskip}3&1&5&5&3&0&2&0\\\noalign{\medskip}2&1&5&5&3&3&0&0
\\\noalign{\medskip}-1-a_{{27}}&2&5&5&3&3&2+a_{{27}}&0\end {array}
 \right].
\]
If $-2 \leq a_{27}\leq -1$, then the matrix $C$ is nonnegative and has spectrum $\sigma$. 
\end{example}

\subsubsection{NIEP for spectrum three positive eigenvalues}
In this section we study the NIEP for three positive eigenvalues and more than or equal three negative eigenvalues.
\begin{theorem}
Consider spectrum $\sigma = \{\lambda_1, \lambda_2, \lambda_3, \lambda_4, \lambda_5, \lambda_6  \}$
such that $\lambda_1 \ge \lambda_2 \ge \lambda_3> 0 > \lambda_4 \ge \lambda_5 \ge \lambda_6$ with nonnegative summation. 
Then there exist a set of  nonnegative $6 \times 6$ matrices that $\sigma $ is its spectrum.
\end{theorem}
{\bf Proof}. 
We consider the matrix $A$ as follow
\[
A=\left[ \begin {array}{cccccc} \lambda_{{1}}&\alpha _{2}+\alpha _{5}+
\alpha _{6}+\alpha _{3}&0&\alpha _{4}&0&0\\ 
0&\lambda_{{2}}&\alpha_{6}+\alpha _{3}& a_{24} &\alpha _{5}&0\\ 
0&0&\lambda_{{3}}& a_{34} &  a_{ 35 } &\alpha _{6}\\
0&0&0&\lambda_{4}&0&0\\ 
0&0&0&0&\lambda_{5}&0\\ 
0&0&0&0&0&\lambda_{6}
\end {array} \right].
\]
The unit lower triangular matrix that solve the problem, we can find as 
\[
L=\left[ \begin {array}{cccccc} 
1&0&0&0&0&0\\ 
1&1&0&0&0&0\\ 
1&1&1&0&0&0\\ 
1&0&0&1&0&0\\
1&1&0&0&1&0\\ 
1&1&1&0&0&1
\end {array} \right]. 
\] 
In this case, the following matrix will be obtained with similarity transformations of the unit lower triangular matrix
\begin{align}
 C=&LAL^{-1}\nonumber \\=&
\begin{small}
\left[ \begin {array}{llllll} 
\lambda_{1}- t &t- \alpha _{4}&0&\alpha _{4}&0&0\\
\lambda_{1}-\lambda_{2}-t -a_{24}&\alpha_{2}+\lambda_{2}&\alpha_{6}+\alpha_{3}&\alpha_{4}+ a_{ 24}&\alpha_{5}&0\\ 
\lambda_{1} -\lambda_{2}-t-a_{24} - a_{34} &\alpha_{2}+\lambda_{2}-\lambda_{3}-a_{35} &\alpha_{3}+\lambda_{
3}&\alpha_{4}+a_{ 24 } + a_{34} &\alpha _{5}+a_{ 35} &\alpha_{6}\\
\lambda_{1} -\lambda_{4} -t&t- \alpha_4 &0&\alpha _{4}+\lambda_{{4}}&0&0\\ 
\lambda_{1} -\lambda_{2}-t-a_{24} &\alpha _{2}+\lambda_{2}-\lambda_{5}&\alpha _{6}+\alpha _{3}&\alpha _{4}+a_{24} &\alpha _{5}+\lambda_{5}&0\\
\lambda_{1}- \lambda_{2}-t -a_{24} - a_{ 34} &\alpha _{2}+\lambda_{2}-\lambda_{3}- a_{ 35} &\alpha_{3}+\lambda_{3}-\lambda_{6}&
\alpha _{4}+ a_{ 24 } + a_{ 34} &\alpha _{5}+a_{ 35 } &\alpha _{6}+
\lambda_{6}\end {array} \right], 
\end{small}
\end{align}
that $t = \sum_{i=2}^{n} \alpha_i$.
This matrix is nonnegative if satisfy the conditions
\begin{align*} 
-\lambda_i & \le \alpha_i, \qquad i= 2,\cdots ,6 ,\nonumber\\
t & \le \lambda_1, \nonumber \\
-\alpha_5 & \le a_{35} \le \alpha_2+\lambda_2-\lambda_3,\\
-\alpha_4 & \le a_{24}  \le \lambda_1-\lambda_2- t, \nonumber\\
-\alpha_4 & \le a_{24}+a_{34} \le \lambda_1-\lambda_2- t,\nonumber
\end{align*} 
 and by similarity 
 $\sigma$ is its spectrum.
\begin{theorem}
Consider spectrum $\sigma = \{\lambda_1, \lambda_2, \cdots,  \lambda_n  \}$
such that $\lambda_1 \ge \lambda_2 \ge \lambda_3> 0 > \lambda_4 \ge \cdots \ge \lambda_n$ with nonnegative summation.
Then there exist a set of  nonnegative $n \times n$ matrix that $\sigma $ is its spectrum.
\end{theorem}
{\bf Proof}. The method of proof this Theorem is continue of process of  previous Theorem and in general in the next subsection, we describe this process.

\subsubsection{The spectrum with $k$ positive eigenvalues}
In this case we study NIEP with $k$ positive eigenvalues and more than or equal $k$ non-positive eigenvalues. One of the most important points that we have to consider is where non-zero values elements  lie in the  unit lower triangular matrix and another important point is the detection and the amount  and locations of nonzero value of the upper triangular matrix $A$ which in  the main  its diagonal located the elements of $\sigma$.
\begin{theorem}
Consider spectrum $\sigma = \{\lambda_1, \lambda_2, \cdots, \lambda_k, \cdots, \lambda_n  \}$
such that $\lambda_1 \ge \lambda_2 \ge \cdots \ge  \lambda_k> 0 \ge \lambda_{k+1} \ge \cdots \ge \lambda_n$ with nonnegative summation.
Then there exist a set of  nonnegative $n \times n$ matrix that $\sigma $ is its spectrum.
\end{theorem}
{\bf Proof.}
For construction of matrix $C$ in which is nonnegative and has eigenvalues $\sigma$, we explain that how to construct  the upper triangular matrix $A$ and  unit lower triangular matrix $L$.  

For construction of the matrix $A$ we put $\lambda_1, \lambda_2, \cdots , \lambda_n$ on main diagonal of this matrix where $\lambda_1, \lambda_2, \cdots , \lambda_k$ are positive and $\lambda_{k+1}, \lambda_{k+2}, \cdots , \lambda_n$ are nonpositive eigenvalues and $k \le \frac{n}{2}$. 

Then we start with the last row and the last column of the matrix $A$ entry $\lambda_n$. We add the value $\alpha_n \ge - \lambda_n$ on the row $k$ and column $n$ of matrix $A$ and  if $ \lambda_k - \alpha_n \ge 0$, we add $\alpha_{n-1}\ge  - \lambda_{n-1}$ on the row $k$ and column $n-1$ and if again  $(\lambda_k - \alpha_n - \alpha_{n-1}) \ge 0$ we add the value $\alpha_{n-2} \ge \lambda_{n-2}$ on the row $k$ and column $n-2$ and continue this way until the first time $\lambda_k -\sum_{i=r}^n \alpha_i <0$, then we add the values $-(\lambda_k -\sum_{i=r}^n \alpha_i $) on the row $k$ and column $k-1$ and we add in row $k$ from column $k+1$ until $r-1$ the values  $\beta_{k,j}$ for $j=k+1, \cdots r-1$ respectively. The values $\alpha_i$ help to us that the effect of the negative eigenvalues in the matrix solution to be eliminated. We select the values $\beta_{k,j}$ to obtain at last nonnegative matrix.    After this step we use the second positive eigenvalue $\lambda_{k-1}$. We add the value $\alpha_{r-1}$ to the row $k-1$ and column $r-1$ and if $\lambda_{k-1} - \alpha_{r-1} >0$ we add the value $\alpha_{r-2} \ge \lambda_{r-2}$  on the row $k-1$ and column $r-2$ and similar the previous step continue. Assume that the index $m$ is the value such that for the first time $\lambda_{k-1} -\sum_{i=k+1}^{m-1} \alpha_i <0$.  In this step again the values  $\beta_{k-1,j}$ started from the column $k+1$ and finished $m-1$. We continue this method from $\lambda_n$ to $\lambda_{k-1}$.  Then we have 
\begin{tiny}
\begin{align*}
&\hspace{-1Cm}A=\\&\hspace{-1Cm}\left[\begin{array}{ccccccccccccccccc}
\lambda_1 & 0 & \cdots & &• & • & • & • & • & • && • & • && • & • & •  \\ 
0 & \lambda_2 & \cdots && • & • & • & • & • & • && • & •& & • & • & •  \\ 
• & • & \ddots & •& & • & • & • & • & • & • & • & • && • && •  \\ 
• & • &  \hspace{0.5Cm}\lambda_s&a_{s,s+1}&& • & 0 & \alpha_{k+1} & • & • & • & •& & •& & • & •  \\ 
• & • &  & \lambda_{s+1}&& • & 0 & \beta_{s+1,k+1} & • & • & • & •& & •& & • & •  \\ 
• & • & • &   \hspace{0.8Cm} \ddots && &\vdots & \vdots& • & • & • & • & •& & &• & • \\ 
• & • & • & • & \lambda_{m-2} &a_{k-2,k-1} &0&\beta_{k-2,k+1}  & \cdots & \alpha_{m-1} & • & • & •& & &• & •  \\ 
• & • & • & • & • & \lambda_{k-1} & a_{k-1,k} & \beta_{k-1,k+1} & \cdots&\beta_{k-1,m-1} & \alpha_m& \cdots& \alpha_{r-1} && • & • & •  \\ 
• & • & • & • & •&  • & \lambda_k & \beta_{k,k+1} & \cdots&\beta_{k,m-1} &\beta_{k,m} &\cdots & \beta_{k,r-1}& \alpha_r & \cdots & \alpha_{n-1} & \alpha_n\\ 
• & • & • & • & •&  • & • & \lambda_{k+1} & • & •& & • & &• & • & • & • \\ 
• & • & • & • & •&  • & • & • & \ddots & • & • & •& & • & • && • \\ 
• & • & • & • & • && • & • & •  & \lambda_{m-1} && • & • && • & • & • \\
• & • & • & • & • && • & • & • & • & \lambda_m & • && • & • & • & • \\ 
• & • & • & • & •& & • & • & • & • & • & \ddots && • & •&  • & • \\
• & • & • & • & •& & • & • & • & • && • & \lambda_{r-1} &  •& & • & • \\
• & • & • & • & • && • & • & • & • && • & • & \lambda_r & • & • & • \\ 
• & • & • & • & • && • & • & • & • && • & • & • & \ddots & • & • \\ 
• & • & • & • & • && • & • & • & • & &• & • & • & • & \lambda_{n-1} & • \\ 
• & • & • & • & • && • & • & • & • && • & • & • & • & • & \lambda_n
\end{array} 
\right],
\end{align*}
\end{tiny}
where 
\begin{align*}
a_{s,s+1}=&-(\lambda_{s-1}-\sum_{i=k+1}^{}\alpha_i),\\
\vdots&\\
a_{k-2,k-1}=&-(\lambda_{k-1}-\sum_{i=m}^{r}\alpha_i),\\
a_{k-1,k}=&-(\lambda_k-\sum_{i=r}^{n}\alpha_i)
\end{align*}
For construction matrix $L$, we act  the following procedure. Whereas $L$ is unit lower triangular matrix we set on the main diagonal of this matrix the number 1. For all $\alpha_j, j=n, \cdots,k+ 1$ that lies on entry $(i,j)$ of matrix $A$, we set the entry $(j,i)$ of the matrix $L$ again 1 and in this row the previous entries will be 1 or appropriate number( that help us for getting nonnegative matrix $C$ in product of three matrices $LAL^{-1}$)  until the column of the last positive $\lambda_i, i=k, k-1,\cdots$ that in the construction the matrix $A$ is used. For all entries of matrix $A$ in which  get positive value for $\lambda_k -\sum_{i=r}^n \alpha_i <0$, we set 1 on the transpose entry of these elements and again in its row the previous elements of these entries will be 1 or appropriate number  until   the column of the last positive that in the construction the matrix $A$ is used. For simplicity, we display these  entries of  matrix $L$ as 1. 
\begin{align*}
L=\begin{array}{l}
• \\ 
• \\ 
• \\ 
\\
s^{th}row \rightarrow\\ 
• \\ 
 \\ 
\\ 
k^{th}row \rightarrow\\ 
\\ 
• \\ • \\ 
\\ 
m^{th}row \rightarrow\\ 
• \\ 
 \\ 
r^{th}row \rightarrow\\
• \\ 
\\
 \\ 
n^{th}row\rightarrow
\end{array} &\left[
\begin{array}{ccccccccccccccccccc}
1 & • & • & • & • & • & • & • & • & • & • & • & • & • & • & • & • &&\\ 
• & 1 & • & • & • & • & • & • & • & • & • & • & • & • & • & • & • &&\\ 
• & • & \ddots & • & • & • & • & • & • & • & • & • & • & • & • & • & • &&\\ 
• & • & • & 1 & • & • & • & • & • & • & • & • & • & • & • & • & •&& \\ 
• & • & • & 1 & 1 & • & • & • & • & • & • & • & • & • & • & • & • &&\\ 
• & • & • & \vdots & • & \ddots & • & • & • & • & • & • & • & • & • & • & •&& \\ 
• & • & • & 1 & 1 & \cdots & 1 & • & • & • & • & • & • & • & • & • & • &&\\ 
• & • & • & 1 & 1 & \cdots & 1 & 1 & • & • & • & • & • & • & • & • & • &&\\ 
• & • & • & 1 & 1 & \cdots & 1 & 1 & 1 & • & • & • & • & • & • & • & • &&\\ 
• & • & • & 1 & 1 &  \cdots & &  &  & 1 & • & • & • & • & • & • & • &&\\ 
• & • & • & \vdots & \vdots & \cdots & \vdots & • & • & • & \ddots & • & • & • & • & • & • &&\\ 
• & • & • & 1 & 1 & \cdots & 1& 0 & • & • & • & 1 & • & • & • & • & • &&\\ 
• & • & • & 1 & 1 & \cdots & 1& 1 & 0 & • & • & • & 1 & • & • & • & • &&\\ 
• & • & • &  \vdots &  \vdots &  \cdots &  \vdots &  \vdots &  \vdots & • & • & • & • & \ddots & • & • & • &&\\ 
• & • & • & 1 & 1 & \cdots & 1 & 1 & 0 & • & • & • & • & • & 1 & • & •&& \\ 
• & • & • & 1& 1 & \cdots & 1 & 1 & 1 & 0 & • & • & • & • & • & 1 & •&& \\ 
• & • & • & \vdots & \vdots &  \cdots & \vdots & \vdots &  \vdots &  \vdots & • & • & • & • & • & • & \ddots&&\\
• & • & • & 1 & 1 & \cdots & 1 & 1 & 1 & 0 & • & • & • & • & • & • & &1&\\
• & • & • & 1 & 1 & \cdots & 1 & 1 & 1 & 0 & • & • & • & • & • & • & &&1
\end{array} 
\right],\\
&\hspace*{.8cm} \begin{array}{cccccccccccccccccccc}
• & • &• &  & \uparrow s^{th}col & • & • & • &   &\uparrow k^{th}col &    &  & • & & • & • & &&
\end{array} 
\end{align*}
where the empty entries have value zero.  By a simple induction we have 
\[
L^{-1} = \left[
\begin{array}{ccccccccccccccccccc}
1 & • & • & • & • & • & • & • & • & • & • & • & • & • & • & • & • &&\\ 
• & 1 & • & • & • & • & • & • & • & • & • & • & • & • & • & • & • &&\\ 
• & • & \ddots & • & • & • & • & • & • & • & • & • & • & • & • & • & • &&\\ 
• & • & • & 1 & • & • & • & • & • & • & • & • & • & • & • & • & •&& \\ 
• & • & • & -1 & 1 & • & • & • & • & • & • & • & • & • & • & • & • &&\\ 
• & • & • & \vdots & • & \ddots & • & • & • & • & • & • & • & • & • & • & •&& \\ 
• & • & • & 1 & 1 & \cdots & 1 & • & • & • & • & • & • & • & • & • & • &&\\ 
• & • & • & 0 & 0 & \cdots & -1 & 1 & • & • & • & • & • & • & • & • & • &&\\ 
• & • & • & 0 & 0 & \cdots & 0 & -1 & 1 & • & • & • & • & • & • & • & • &&\\ 
• & • & • & 1 & 1 &  \cdots & &  &  & 1 & • & • & • & • & • & • & • &&\\ 
• & • & • & \vdots & \vdots & \cdots & \vdots & • & • & • & \ddots & • & • & • & • & • & • &&\\ 
• & • & • & 1 & 1 & \cdots & 1& 0 & • & • & • & 1 & • & • & • & • & • &&\\ 
• & • & • & 1 & 1 & \cdots & 1& 1 & 0 & • & • & • & 1 & • & • & • & • &&\\ 
• & • & • &  \vdots &  \vdots &  \cdots &  \vdots &  \vdots &  \vdots & • & • & • & • & \ddots & • & • & • &&\\ 
• & • & • & 1 & 1 & \cdots & 1 & 1 & 0 & • & • & • & • & • & 1 & • & •&& \\ 
• & • & • & 1& 1 & \cdots & 1 & 1 & 1 & 0 & • & • & • & • & • & 1 & •&& \\ 
• & • & • & \vdots & \vdots &  \cdots & \vdots & \vdots &  \vdots &  \vdots & • & • & • & • & • & • & \ddots&&\\
• & • & • & 1 & 1 & \cdots & 1 & 1 & 1 & 0 & • & • & • & • & • & • & &1&\\
• & • & • & 1 & 1 & \cdots & 1 & 1 & 1 & 0 & • & • & • & • & • & • & &&1
\end{array} 
\right]
\]
\begin{remark}
It is important to note that   a given list of eigenvalues $\sigma$ is realizable with our  method, when  the value  of Perron eigenvalue or the last positive eigenvalue that is used for difference between positive eigenvalues  $\lambda_s$ and $\sum \alpha_i$ at least equal the minimum of negative eigenvalues.  For instance the list  $\sigma=\{6,1,1,1,-4,-4\}$ is solvable and the list $\sigma=\{6,1,1,1,1,-4,-4\}$ is not solvable by our method, although we add a positive amount in this list, since in the solving process we can not use the Perron eigenvalue for the last eigenvalue -4.  When  we put the elements of main diagonal of upper triangular matrix $A$ from  up to down by decreasing order,   but if we  interchange the elements of main diagonal of matrix $A$ as the list  $\sigma=\{6,1,1,1,-4,-4,1\}$ is solvable and we do not consider the last column in our algorithm. So far for the convenience, we assume that  the values of main diagonal of matrix $A$ is decreasing and in general   it is not necessary. 
\end{remark}

\subsection{ The spectrum with $k > \frac{n}{2}$}
From the beginning discussion until now we study NIEP with  the number of negative eigenvalues more than or equal the number of positive eigenvalues. In this section we study in which the number of positive eigenvalues  in the  given spectrum $ \sigma $ more than the  number of negative eigenvalues.

The spectrum with one negative eigenvalues is very simply solved. In continue we study the cases that the spectrum $ \sigma $ has more than or equal two negative eigenvalues.
So there are at least $ 5 $ members in $ \sigma$. Unfortunately in this section there exist some cases that the lower triangular matrix must have some non zero-one values. 

\subsubsection{The spectrum  with two negative eigenvalues}
 In this subsection we solve NIEP with three  positive and  two negative eigenvalues that satisfies in condition \eqref{JLL} and then we solve the extension of problem of NIEP with three positive eigenvalues and more than three negative eigenvalues. 
 
\begin{theorem}
Let $ \sigma = \{ \lambda_1, \lambda_2, \lambda_3, \lambda_4, \lambda_5  \} $ such that $ \lambda_1\ge \lambda_2\ge \lambda_3\ge 0 >  \lambda_4\ge \lambda_5$ and satisfies in \eqref{JLL}, then three exist a set of nonnegative eigenvalues in which $ \sigma $ is its spectrum.
\end{theorem}
\begin{proof}
In this case we construct the members $ A $ and $ L $ as follows.

\[
A=\left[ \begin {array}{ccccc} \lambda_{{1}}&-\lambda_{{2}}-\lambda_{{4
}}-\lambda_{{3}}-\lambda_{{5}}&0&0&0\\\noalign{\medskip}0&\lambda_{{2}
}&-\lambda_{{3}}-\lambda_{{5}}&-\lambda_{{4}}&0\\\noalign{\medskip}0&0
&\lambda_{{3}}&a_{{34}}&-\lambda_{{5}}\\\noalign{\medskip}0&0&0&
\lambda_{{4}}&0\\\noalign{\medskip}0&0&0&0&\lambda_{{5}}\end {array}
 \right],
\qquad
L=\left[ \begin {array}{ccccc} 1&0&0&0&0\\\noalign{\medskip}1&1&0&0&0
\\\noalign{\medskip}l_{{31}}&1&1&0&0\\\noalign{\medskip}1&1&0&1&0
\\\noalign{\medskip}l_{{51}}&1&1&0&1\end {array} \right] ,
\]
so   
for simplicity, we use the following symbols to represent the matrix $C$:
\begin{align*}
C&= LA L^{-1} =\left[ \begin {array}{ccccc} s_1& \lambda_{1}-s_1 &0&0&0\\
 c_{2 1} &0&-(\lambda_{{3}}+\lambda_{{5}})&-
\lambda_{{4}}&0\\
c_{31} &c_{32}&0&-\lambda_{{4}}+a_{{34}}&-\lambda_{{5}}\\
c_{41} &-\lambda_{{4}}&-(\lambda_{{3}}+
\lambda_{{5}})&0&0\\ 
 c_{51} &c_{52}&-\lambda_
{{5}}&-\lambda_{{4}}+a_{{34}}&0\end {array} \right], 
\end{align*}
where
\begin{align*}
c_{21}=c_{41}& = s_1- \lambda_{2}-  (  \lambda_{3}+\lambda_{5})  ( 1-l_{ 31 }  ),\\
c_{31}&=  l_{{31}} s_1 - \lambda_{{2}}-\lambda_{{5}} ( 1-l_{{51}}  ), \\
c_{51}&= l_{{51}} s_1 - \lambda_{{2}}-\lambda_{{5}} \left( 1-l_{{31}} \right), \\
c_{32}&=l_{ 31 } ( \lambda_{1}-s_1) +s_1-\lambda_1-\lambda_3-a_{{34}},\\
c_{52}&=l_{ 51 } ( \lambda_{1}-s_1) +s_1-\lambda_1-\lambda_3-a_{{34}}.
\end{align*}
 The matrix $ C $ is nonnegative if all entries of this matrix will be nonnegative and then we have
 \begin{align*}
 1-\frac{\lambda_3 + a_{34}}{s_1-\lambda_1} \le & l_{31} \le 1- \frac{s_1 - \lambda_2}{\lambda_3 +\lambda_5},\\
  1-\frac{\lambda_3 + a_{34}}{s_1-\lambda_1} \le & l_{51} \le 1- \frac{l_{31} s_1 - \lambda_2}{\lambda_5},\\
 \lambda_4 \le & a_{34 } \le \frac{(s_1 - \lambda_1)(s_1 - \lambda_2)}{\lambda_3 +\lambda_5}- \lambda_3.
 \end{align*}
\end{proof}
\begin{remark}
If the number of positive eigenvalues more than three, we can continue the above method for construction of matrices $A$ and $L$ and get nonnegative matrix $C$, for instance if  for $n=6$ in spectrum $\sigma$ we have $ \lambda_1\ge \lambda_2\ge \cdots \ge  \lambda_{n-2}\ge 0 >  \lambda_{n-1}\ge \lambda_n $,  then we assume that the matrices $A$ and $L$ as
\[
A=\left[ \begin {array}{cccccc} \lambda_{{1}}&-\lambda_{{3}}-\lambda_{{
4}}-\lambda_{{5}}-\lambda_{{6}}-\lambda_{{2}}&0&0&0&0
\\ \noalign{\medskip}0&\lambda_{{2}}&-\lambda_{{3}}-\lambda_{{4}}-
\lambda_{{5}}-\lambda_{{6}}&0&0&0\\ \noalign{\medskip}0&0&\lambda_{{3}
}&-\lambda_{{4}}-\lambda_{{6}}&-\lambda_{{5}}&0\\ \noalign{\medskip}0&0
&0&\lambda_{{4}}&\alpha _{45}&-\lambda_{{6}}\\ \noalign{\medskip}0&0&0
&0&\lambda_{{5}}&0\\ \noalign{\medskip}0&0&0&0&0&\lambda_{{6}}
\end {array} \right],
\]
\[ 
 L=\left[ \begin {array}{cccccc} 1&0&0&0&0&0\\ \noalign{\medskip}1&1&0&0
&0&0\\ \noalign{\medskip}1&1&1&0&0&0\\ \noalign{\medskip}l_{41}&l_{42}
&1&1&0&0\\ \noalign{\medskip}1&1&1&0&1&0\\ \noalign{\medskip}l_{61}&l_
{62}&1&1&0&1\end {array} \right],
\]
 therefore the nonnegative matrix $C$ obtain as follows 
\[
 C=\left[ \begin {array}{cccccc} s_1&\lambda_{1}-s_1&0&0&0&0\\ 
s_1 - \lambda_{2}&0& \lambda_{1}+\lambda_{2}-s_1&0&0&0\\ 
c_{31} &c_{32}&0&-\lambda_{{4}}-\lambda_{{6}}&-\lambda_{{5}}&0\\ 
c_{41}& c_{42} & c_{43}&0&-
\lambda_{{5}}+\alpha _{45}&-\lambda_{{6}}\\
c_{51} &c_{52} &-\lambda_{{5}}&-\lambda_{{4}}-\lambda_{{6}}&0&0
\\  
c_{61} &c_{62}&c_{62}&-\lambda_{{6}}&-\lambda_{{5}}
+\alpha _{45}&0\end {array} \right], 
 \] 
where
\begin{align*}
c_{31}&=s_1 - \lambda_2 + \left( -\lambda_{{4}}-\lambda_{{6}} \right)  \left( -l_{41}+l_{42} \right),\\
c_{41}&=l_{41}s_1 -l_{42}\lambda_{{2}} -\lambda_{{6}} \left( -l_{61}+l_{62} \right), \\
c_{51} & =s_1 - \lambda_{2}+ \left( -\lambda_{{4}}-\lambda_{{6}} \right)  \left( -l_{41}+l_{42} \right),\\
c_{61}& = l_{61}s_1 -l_{62}\,\lambda_{{2}}-\lambda_{{6}} \left( -l_{41}+l_{42} \right),\\
c_{32}&= -\lambda_{{3}}+ \left( -\lambda_{{4}}-\lambda_{{6}} \right)  \left( 1-l_{42} \right), \\
c_{42}&= \left(  \lambda_{1}-s_1\right) (l_{41} - l_{42})-\lambda_{{3}}-\lambda_{{6}} \left( 1-l_{62} \right), \\
c_{52}& = -\lambda_{{3}}+ \left( -\lambda_{{4}}-\lambda_{{6}} \right)  \left( 1-l_{42} \right),\\
c_{62}& =l_{62}\left( \lambda_{1}+\lambda_{2}-s_1 \right) +\lambda_{{3}}+
\lambda_{{6}}+\lambda_{{5}}-\alpha _{45},\\
c_{43}& = l_{42} \left(  \lambda_{1}+\lambda_{2}-s_1\right) +\lambda_{{3}}+\lambda_{{6}}+\lambda_{{5}}-\alpha _{45},\\
c_{63}&= \left(  \lambda_{1}-s_1 \right) (l_{61} - l_{62}) -\lambda_{{3}}-\lambda_{{6}} \left( 1-l_{42} \right).\\
\end{align*}
In general case we can construct the matrices $A, L$ and $C$ as above method and the  non zero-one entries of matrix $L$
is located  in the row number of the last positive eigenvalues  and row number of the last negative eigenvalues of matrix $A$.
\end{remark}
\begin{example}
Let $ \sigma =  \{ 6,1,1,-4,-4 \} $. This spectrum is solved in \cite{5} and  in \cite{borobia} is discussed about its C-realziabity.  We can find the nonnegative matrix $C$ that has spectrum $\sigma $  by our method and simpler. For this
we consider matrices $A$ and $L$ as:
\[
\left[ \begin {array}{ccccc} 6&6&0&0&0\\\noalign{\medskip}0&1&3&4&0
\\\noalign{\medskip}0&0&1&-4&4\\\noalign{\medskip}0&0&0&-4&0
\\\noalign{\medskip}0&0&0&0&-4\end {array} \right],
\qquad
\left[ \begin {array}{ccccc} 1&0&0&0&0\\\noalign{\medskip}1&1&0&0&0
\\\noalign{\medskip}l_{{21}}&1&1&0&0\\\noalign{\medskip}1&1&0&1&0
\\\noalign{\medskip}l_{{51}}&1&1&0&1\end {array} \right] 
\]
then the matrix $C$ is computed by relation $C=LAL^{-1}$ i.e.
\[
\left[ \begin {array}{ccccc} 0&6&0&0&0\\\noalign{\medskip}2-3\,l_{{21
}}&0&3&4&0\\\noalign{\medskip}3-4\,l_{{51}}&-3+6\,l_{{21}}&0&0&4
\\\noalign{\medskip}2-3\,l_{{21}}&4&3&0&0\\\noalign{\medskip}3-4\,l_{{
21}}&-3+6\,l_{{51}}&4&0&0\end {array} \right].
\]
The interval that of $l_21$ and $l_{51}$ that by them them matrix $C$ is nonnegative very simple is computed. 
\end{example}
\subsubsection{The spectrum  with three negative eigenvalues}
\begin{theorem}
Assume  that $ \sigma= \{ \lambda_1, \lambda_2, \cdots, \lambda_n \} $ where $ \lambda_1 \ge \lambda_2 \ge \cdots \ge \lambda_{n-3} \ge 0 > \lambda_{n-2} > \lambda_{n-1} > \lambda_n $ with $ \sum_{i=1}^{n} \lambda_i \ge 0$, then there exist a set of nonnegative matrix $C$ such that $\sigma $ is its spectrum.
\end{theorem}
\begin{proof}. The proof process is the same as before, and we only need to give an example to illustrate this process.
\end{proof}
\begin{example}
Let $\sigma =\{\lambda_1,\lambda_2,\lambda_3,\lambda_4,\lambda_5,\lambda_6,\lambda_7,\lambda_8 \}=\{8,2,2,2,1,-5,-5,-5 \}$ that $\sigma$ satisfies in conditions \eqref{JLL} special $\sum \lambda_i = \sum \lambda_i^3 =0$, and $\sum_{k=4} \lambda_i^k > 0$.
By our method we consider
 $$
 A=\left[ \begin {array}{cccccccc} \lambda_{{1}}&-s_1+\lambda_1 &0&0&0&0&0&0\\\noalign{\medskip}0&\lambda_{{2}}&-s_1+\lambda_1+\lambda_2&0
&0&0&0&0\\\noalign{\medskip}0&0&\lambda_{{3}}&-\lambda_{{8}}-\lambda_{
{5}}-\lambda_{{4}}-\lambda_{{7}}&0&-\lambda_{{6}}&0&0
\\\noalign{\medskip}0&0&0&\lambda_{{4}}&-\lambda_{{8}}-\lambda_{{5}}&a
_{{46}}&-\lambda_{{7}}&0\\\noalign{\medskip}0&0&0&0&\lambda_{{5}}&a_{{
56}}&a_{{57}}&-\lambda_{{8}}\\\noalign{\medskip}0&0&0&0&0&\lambda_{{6}
}&0&0\\\noalign{\medskip}0&0&0&0&0&0&\lambda_{{7}}&0
\\\noalign{\medskip}0&0&0&0&0&0&0&\lambda_{{8}}\end {array} \right], 
 $$
 where $s_1=\sum_{i=1}^n\lambda_i$ and
  $$
  L= \left[ \begin {array}{cccccccc} 1&0&0&0&0&0&0&0\\\noalign{\medskip}1&
1&0&0&0&0&0&0\\\noalign{\medskip}l_{{31}}&1&1&0&0&0&0&0
\\\noalign{\medskip}l_{{41}}&l_{{42}}&1&1&0&0&0&0\\\noalign{\medskip}l
_{{51}}&l_{{52}}&l_{{53}}&1&1&0&0&0\\\noalign{\medskip}l_{{61}}&l_{{62
}}&1&0&0&1&0&0\\\noalign{\medskip}l_{{71}}&l_{{72}}&l_{{73}}&1&0&0&1&0
\\\noalign{\medskip}l_{{81}}&l_{{82}}&l_{{83}}&l_{{84}}&1&0&0&1
\end {array} \right]. 
  $$
  If we select $a_{46}=-5,a_{56}=\frac{20}{7}, a_{57}=-5$ and also select $l_{31}=\frac45, l_{41}=\frac14, l_{42}=\frac{33}{70}, l_{51}=\frac{1}{20}, l_{61}=\frac{23}{70}, l_{62}=\frac12, l_{71}=l_{41}, l_{72}=l_{42}, l_{81}=l_{51}, l_{82}=\frac{9}{70},\l_{83}=\frac37 $ the matrix $C$ is solution of problem is computed as follows:
  \[
  C=\left[ \begin {array}{cccccccc} 0&8&0&0&0&0&0&0\\\noalign{\medskip}0&0
&10&0&0&0&0&0\\\noalign{\medskip}{\frac {57}{140}}&{\frac {13}{5}}&0&7
&0&5&0&0\\\noalign{\medskip}{\frac {67}{140}}&\frac{1}{14}&0&0&4&0&5&0
\\\noalign{\medskip}{\frac {19}{140}}&{\frac {1}{70}}&0&0&0&0&0&5
\\\noalign{\medskip}{\frac {11}{20}}&{\frac {23}{70}}&0&7&0&0&0&0
\\\noalign{\medskip}{\frac {67}{140}}&\frac{1}{14}&0&5&4&0&0&0
\\\noalign{\medskip}{\frac {19}{140}}&{\frac {1}{70}}&0&0&5&0&0&0
\end {array} \right]. 
  \]
 Now we present  another solution for this problem. If we consider the matrix 
 \[
 A=\left[ \begin {array}{cccccccc} \lambda_{{1}}&-\lambda_{{7}}-\lambda_
{{3}}-\lambda_{{2}}&0&-\lambda_{{8}}-\lambda_{{5}}-\lambda_{{4}}&0&-
\lambda_{{6}}&0&0\\\noalign{\medskip}0&\lambda_{{2}}&-\lambda_{{7}}-
\lambda_{{3}}&a_{{24}}&a_{{25}}&a_{{26}}&0&0\\\noalign{\medskip}0&0&
\lambda_{{3}}&a_{{34}}&a_{{35}}&a_{{36}}&-\lambda_{{7}}&0
\\\noalign{\medskip}0&0&0&\lambda_{{4}}&-\lambda_{{8}}-\lambda_{{5}}&a
_{{46}}&a_{{47}}&0\\\noalign{\medskip}0&0&0&0&\lambda_{{5}}&a_{{56}}&a
_{{57}}&-\lambda_{{8}}\\\noalign{\medskip}0&0&0&0&0&\lambda_{{6}}&0&0
\\\noalign{\medskip}0&0&0&0&0&0&\lambda_{{7}}&0\\\noalign{\medskip}0&0
&0&0&0&0&0&\lambda_{{8}}\end {array} \right], 
 \]
  and the matrix
  \[
  L=\left[ \begin {array}{cccccccc} 1&0&0&0&0&0&0&0\\\noalign{\medskip}1&
1&0&0&0&0&0&0\\\noalign{\medskip}l_{{31}}&1&1&0&0&0&0&0
\\\noalign{\medskip}1&0&0&1&0&0&0&0\\\noalign{\medskip}l_{{51}}&l_{{52
}}&l_{{53}}&1&1&0&0&0\\\noalign{\medskip}1&0&0&0&0&1&0&0
\\\noalign{\medskip}l_{{71}}&l_{{72}}&1&0&0&0&1&0\\\noalign{\medskip}l
_{{81}}&l_{{82}}&l_{{83}}&l_{{84}}&1&0&0&1\end {array} \right], 
  \]
  again if we select $a_{24}=-2,a_{34}=a_{25}=a_{35}=a_{47}=a_{57}=0, a_{46}=-5,a_{56}=-2,a_{26}=a_{26}=-5$ and $l_{31}=l_{51}=2, l_{52}=\frac14, l_{53}=0, l_{71}=\frac43,l_{72}=\frac{-2}{3},l_{81}=2,l_{82}=l_{83}=0,l_{84}=1$, then the nonnegative matrix 
  \[
  C=\left[ \begin {array}{cccccccc} 0&1&0&2&0&5&0&0\\\noalign{\medskip}2&0
&3&0&0&0&0&0\\\noalign{\medskip}0&{\frac {22}{3}}&0&2&0&0&5&0
\\\noalign{\medskip}0&0&0&0&4&0&0&0\\\noalign{\medskip}1/2&7/4&3/4&1/2
&0&7/4&0&5\\\noalign{\medskip}5&1&0&2&0&0&0&0\\\noalign{\medskip}5/3&0
&0&4&0&5&0&0\\\noalign{\medskip}5/4&3/4&0&1&5&3&0&0\end {array}
 \right] ,
  \]
  is solution of problem.
\end{example}
\subsubsection{The spectrum  with $k$ negative eigenvalues}
In this subsection we consider the extended of problem.
\begin{theorem}
Consider $ \sigma = \{ \lambda_1, \lambda_2, \cdots , \lambda_n \} $ where $ \lambda_1 \ge \lambda_2 \ge \cdots \ge \lambda_{n-k} \ge 0 > \lambda_{n-k+1} \ge \cdots \ge \lambda_k $,
then there exist a set of nonnegative matrices that $ \sigma $ is its spectrum.
\end{theorem}
\begin{proof}
For construction of matrix $A$ and $C$ we will do as follows.\\
In main diagonal elements of upper triangular matrix $ A $, the all elements of $ \sigma$ in order $ \lambda_1 \ge \lambda_2 \ge \cdots \ge \lambda_n  $, we put the entries $ (n-k, n), (n-k-1, n-1), \cdots , (n-2k, n-k) $ the real numbers $ \alpha_i \ge -\lambda_i $ that $ i= n-k, \cdots, n-2k $ respectively. In addition the entries $ (n-j-1, n-j) $ for $ j=k, \cdots 2k-1 $ we set $ -(\alpha_{n-j} -\lambda_{n-j}) $ respectively,
and another element between rows $1$ and $n-k$ and between $\alpha_i$ and $ \lambda_i $ we put appreciate real numbers whereas the matrix $ C= LAL^{-1} $ be nonnegative.\\
We choose the matrix $ L $  a lower unit triangular matrix and in the last row of $ L $ we consider $ 0 \le l_{r1}, l_{r2}, \cdots, l_{rk} $ that $ \sum_{i=1}^{k} l_{ni} =k $
and element of row $ n-1 $ all be $ 1 $ for $ j=1, 2, \cdots, k-1$, and row $ n-2 $ all element from $ j=1$, to $ n-2 $ be $ 1 $, and the row $ n-k $ we have $ n-k+1 $ element $ 1 $ and another element be zero.
We repeat above method with  less that one element for row $n-k-1$, i.e. in row $n-k-1$  we set $0 \le l_{n-k-1, 1}, l_{n-k-1, 2}, \cdots , l_{n-k-1, n-k-1}$, with $\sum_{j=1}^{k-1}l_{n-k-1,j}=k-1$
with choosing above method then the matrix $ C = LA L^{-1} $ is nonnegative. 
\end{proof}

\section{Complex spectrum}
In this section we assume that $\sigma \in \mathbb{C}$. We recall that  $\sigma = \bar{\sigma}$ is necessary condition for solvable problem. At first we consider the case $n=3$, this case has two complex conjugate eigenvalues and a Perron eigenvalue $\lambda_1 \ge | \lambda_2 \pm i \mu_2 | $ and for $n=4$ we have two real eigenvalues and a pair complex conjugate eigenvalue. For $n=5$ or more then we have several case that in continue we study all of them.

In this case we introduce the complex unit lower triangular matrix as follows:

The elements of main diagonal unit lower triangular matrix $L$ in this case are $ i$ or $1$.

\subsection{The case $n=3$}
We consider this case in two parts. At first we assume that $\lambda_2 > 0$ in $\sigma = \{ \lambda_1, \lambda_2\pm i \mu_2 \}$   and in another case $\lambda_2 <0$.

\textit{\textbf{Case 1. $\lambda_2>0$}}:
Now we bring   
an important Theorem from \cite{soto} and by helping unit lower triangular matrix in complex case, we find a nonnegative matrix for this Theorem.
\begin{theorem}
Let σ$\sigma=\{\lambda_1, -\lambda_2 \pm i \mu_2\} $, with $ \lambda_2>0 $, $ \mu_2>0  $ and $\lambda_1 \ge \sqrt{\lambda_2^2+\mu_2^2}  $. Then, $ \sigma $ is realizable by a nonnegative matrix $ A $ if and only if
\[\lambda_1 \ge 2 \lambda_2 + 3 \max \{0, \frac{\mu_2}{\sqrt{3}}-\lambda_2\}.\]
\end{theorem}
 Now we present solution of this Theorem by our method. We consider the matrix $A$ and $L$ as follows:
$$
A= \left[ \begin {array}{ccc} 
\lambda _{1}&i \mu _{2}&0\\
0&\lambda _{2}-i\mu _{2}&-i\mu _{2}\\
0&0&\lambda _{2}+i\mu _{2}\end {array} \right], \qquad
L=\left[ \begin {array}{ccc} 1&0&0\\
l_{{21}}&i&0\\
1&1&1\end {array} \right], 
$$
then  we have 
$$
C=\left[ \begin {array}{ccc}
-l_{21}\mu _{2}+\lambda_{1}&\mu _{2}&0\\
  -\mu _{2}({l_{21}}^{2}+1)+l_{21}(\lambda _{1}-\lambda _{2})&l_{21}\mu _{2}+\lambda _{2}&\mu _{2}\\  
  -\lambda _{2}+\lambda _{1}&0&\lambda _{2}
\end {array} \right],
$$
 and if $\frac{(\lambda_1-\lambda_2)-\sqrt{(\lambda_1-\lambda_2)^2-4\mu_2^2}}{2\mu_2}\le l_{21}\le  \frac{(\lambda_1-\lambda_2)+\sqrt{(\lambda_1-\lambda_2)^2-4\mu_2^2}}{2\mu_2}$ then $C$ is real nonnegative matrix. 
 
 \textit{\textbf{Case 2. $\lambda_2<0$}}:
 In this case the matrices $A$ and $L$ present as 
 $$
 A=\left[ \begin {array}{ccc} \lambda _{1}&i\mu _{2}&0
\\  0&\lambda _{2}-i\mu _{2}&-i\mu _{2}
\\  0&0&\lambda _{2}+i\mu _{2}\end {array} \right], 
\qquad
L=\left[ \begin {array}{ccc} 1&0&0\\ \noalign{\medskip}l_{21}&i&0
\\ \noalign{\medskip}{\frac {{\lambda _{2}}^{2}}{{\mu _{2}}^{2}}}+1&{
\frac {-\lambda _{2}}{\mu _{2}}}i+1&1\end {array} \right],  
 $$
 then the matrix $C=LAL^{-1}$  find as 
 $$
 \left[ \begin {array}{ccc} -l_{21}\,\mu _{2}+\lambda _{1}&\mu _{2}&0
\\ \noalign{\medskip}-{\frac {{\mu _{2}}^{2}{l_{21}}^{2}-l_{21}\,
\lambda _{1}\,\mu _{2}+2\,\mu _{2}\,l_{21}\,\lambda _{2}+{\mu _{2}}^{2
}+{\lambda _{2}}^{2}}{\mu _{2}}}&l_{21}\,\mu _{2}+2\,\lambda _{2}&\mu 
_{2}\\ \noalign{\medskip}{\frac { \left( {\mu _{2}}^{2}+{\lambda _{2}}
^{2} \right) \lambda _{1}}{{\mu _{2}}^{2}}}&0&0\end {array} \right].
 $$
 If $l_{21}$ lies in the interval 
 $$
  \frac {\lambda _{1}-2\,\lambda _{2}-\sqrt {-4\,{\mu _{2}}^{2}+{
\lambda _{1}}^{2}-4\,\lambda _{1}\,\lambda _{2}}}{2\mu _{2}}
\le l_{21}
  \le  \frac {\lambda _{1}-2\,\lambda _{2}+\sqrt {-4\,{\mu _{2}}^{2}+{
\lambda _{1}}^{2}-4\,\lambda _{1}\,\lambda _{2}}}{2\mu _{2}},
$$
 then the matrix $C$ is nonnegative. 
 
 \subsection{The case n=4}
 The case $n=4$ very similar to the  previous case. Let $\sigma = \{ \lambda_1, \lambda_2, \lambda_3 \pm \mu_3 i\}$ with nonnegative summation.
 If $\lambda_2>0$ and $\lambda_1$ is the Perron eigenvalue, then by the  cases 2 of subsection 6.1,  we construct the $3 \times 3$ nonnegative  matrix  $C_1$ and  we add the value of $\lambda_2$  on the main diagonal  and find  $4 \times 4$ matrix    $C= \left ( \begin{array}{cc}
C_1 & 0 \\ 
0 & \lambda_2
\end{array} \right ).$ 
 The matrix $C$ is solution of problem and has spectrum $\sigma$.
 
  If $\lambda_2 <0$, 
 we construction the matrix $C$, by combination of subsection 6.1 and some of the related real above sections i.e. we add the value $-\lambda_2$  in the first row of matrix $A$ in the same column that the value  $\lambda_2$ lies on the main diagonal of matrix $A$, and the another entries of matrix $A$ and the matrix $L$ are the combination of  above real sections and  above complex subsection is constructed.
 \begin{example}
 Consider $\sigma=\{6, -2, -2-i, -2 +i\}$. Since the sum of $\lambda_i$ equal zero, then by our method we find a nonnegative matrix with zero trace. For this we consider two matrices $A$ and $L$ as follows:
 \[
 A=\left[ \begin {array}{cccc} 6&2&i&0\\\noalign{\medskip}0&-2&0&0
\\\noalign{\medskip}0&0&-2-i&-i\\\noalign{\medskip}0&0&0&-2+i
\end {array} \right], \qquad
L=\left[ \begin {array}{cccc} 1&0&0&0\\\noalign{\medskip}1&1&0&0
\\\noalign{\medskip}l_{{31}}&0&i&0\\\noalign{\medskip}5&0&1+2\,i&1
\end {array} \right], 
 \]
 then the matrix $C$ equals 
 \[
 C=LAL^{-1}=\left[ \begin {array}{cccc} 4-l_{{31}}&2&1&0\\\noalign{\medskip}6-l_{
{31}}&0&1&0\\\noalign{\medskip}8\,l_{{31}}-{l_{{31}}}^{2}-5&2\,l_{{31}
}&l_{{31}}-4&1\\\noalign{\medskip}20&10&0&0\end {array} \right],
 \]
 if $4-\sqrt{11} \leq l_{31} \leq 4$, then the matrix $C$ is solution of problem.
 \end{example}
 \subsection{The case $n \ge 5$} For $n\ge 5$ we combine the  above complex case  and above real cases and find solution. We will give  some examples that already are solved and again solve them by our method. 
  \begin{example} \cite{Robiano2018} Let $\sigma=\{12, i\sqrt{3}, -i\sqrt{3} , 4+3i, 4-3i\}$, with nonnegative 
  real part of complex eigenvalues,  then there exist a nonnegative matrix that $\sigma$ is spectrum.
  
For solving this problem we consider two m  matrices $A$ and $L$  as follows:
\[
A=\left[ \begin {array}{ccccc} 12&i\sqrt {3}&0&3\,i&0
\\\noalign{\medskip}0&-i\sqrt {3}&-i\sqrt {3}&0&0\\\noalign{\medskip}0
&0&i\sqrt {3}&0&0\\\noalign{\medskip}0&0&0&4-3\,i&-3\,i
\\\noalign{\medskip}0&0&0&0&4+3\,i\end {array} \right], \qquad
L=\left[ \begin {array}{ccccc} 1&0&0&0&0\\\noalign{\medskip}l_{{21}}&i&0
&0&0\\\noalign{\medskip}1&1&1&0&0\\\noalign{\medskip}l_{{41}}&0&0&i&0
\\\noalign{\medskip}1&0&0&1&1\end {array} \right], 
\]
then the following   matrix $C$, 
\[
C=\left[ \begin {array}{ccccc} 12-l_{{21}}\sqrt {3}-3\,l_{{41}}&\sqrt {
3}&0&3&0\\\noalign{\medskip}-{l_{{21}}}^{2}\sqrt {3}+ \left( 12-3\,l_{
{41}} \right) l_{{21}}-\sqrt {3}&l_{{21}}\sqrt {3}&\sqrt {3}&3\,l_{{21
}}&0\\\noalign{\medskip}12-3\,l_{{41}}&0&0&3&0\\\noalign{\medskip}8\,l
_{{41}}-l_{{41}}\sqrt {3}l_{{21}}-3-3\,{l_{{41}}}^{2}&l_{{41}}\sqrt {3
}&0&3\,l_{{41}}+4&3\\\noalign{\medskip}8-l_{{21}}\sqrt {3}&\sqrt {3}&0
&0&4\end {array} \right].
\]
is nonnegative and has spectrum $\sigma$, if satisfies the  conditions 
\begin{align*}
 &0\le l_{41}\le 4-\frac{2}{3}\sqrt{3}, \\
 &-1/2\,\sqrt {3}l_{41}+2\,\sqrt {3}-1/2\,\sqrt {3\,{l_{41}}^{2}-24\,l_{
41}+44}
 \le l_{21}, \\
&  -1/2\,\sqrt {3}l_{41}+2\,\sqrt {3}+1/2\,\sqrt {3\,{l_{41}}^{2}-24\,l_{
41}+44}\ge l_{21}.
\end{align*}
\end{example}
\begin{example} Let $\sigma=\{6, -2-3i , -2 +3i , -1 -i , -1 +i\}$, with negative 
  real part of complex eigenvalues,  then there exist a nonnegative matrix that $\sigma$ is spectrum.
  
   For this problem the matrices $A$ and $L$ will be as:
  \[
  A=\left[ \begin {array}{ccccc} 6&3\,i&0&i&0\\\noalign{\medskip}0&-2-3\,
i&-3\,i&0&0\\\noalign{\medskip}0&0&-2+3\,i&0&0\\\noalign{\medskip}0&0&0
&-1-i&-i\\\noalign{\medskip}0&0&0&0&-1+i\end {array} \right],
L=\left[ \begin {array}{ccccc} 1&0&0&0&0\\\noalign{\medskip}l_{{21}}&i&0
&0&0\\\noalign{\medskip}{\frac {13}{9}}&1+2/3\,i&1&0&0
\\\noalign{\medskip}l_{{41}}&0&0&i&0\\\noalign{\medskip}2&0&0&1+i&1
\end {array} \right] 
  \]
  then the matrix $C$ is computed as:
  \[
 C=LAL^{-1} =\left[ \begin {array}{ccccc} 6-3\,l_{{21}}-l_{{41}}&3&0&1&0
\\\noalign{\medskip}10\,l_{{21}}-3\,{l_{{21}}}^{2}-13/3-l_{{21}}l_{{41
}}&3\,l_{{21}}-4&3&l_{{21}}&0\\\noalign{\medskip}{\frac {26}{3}}-{
\frac {13}{9}}\,l_{{41}}&0&0&{\frac {13}{9}}&0\\\noalign{\medskip}8\,l
_{{41}}-3\,l_{{21}}l_{{41}}-2-{l_{{41}}}^{2}&3\,l_{{41}}&0&l_{{41}}-2&
1\\\noalign{\medskip}12-6\,l_{{21}}&6&0&0&0\end {array} \right].
  \]
  To obtain the nonnegative matrix  $C$ it is necessary that  $l_{21} =\frac43$ and $l_{41}=2$, then the matrix $C$ is nonnegative and has spectrum $\sigma$  and the values of $l_{21}$ and $ l_{41}$ are unique in this example. 
  \end{example}
\bigskip
\noindent
\bibliographystyle{amsplain}

\vspace{0.1in}
\hrule width \hsize \kern 1mm
\end{document}